\documentclass[leqno]{amsart}
\usepackage{amsfonts,amssymb,amsmath,amsgen,amsthm}
\usepackage{hyperref,color}
\usepackage{pdfsync}

\theoremstyle{plain}
\newtheorem{theorem}{Theorem}[section]
\newtheorem{definition}[theorem]{Definition}

\newtheorem{lemma}[theorem]{Lemma}
\newtheorem{corollary}[theorem]{Corollary}
\newtheorem{proposition}[theorem]{Proposition}
\newtheorem{hyp}[theorem]{Assumption}
\theoremstyle{remark}
\newtheorem{remark}[theorem]{Remark}

\def\dis
{\displaystyle}

\def\R{{\mathbf R}}
\def\T{{\mathbf T}}
\def\Z{{\mathbf Z}}
\def\Sch{{\mathcal S}}
\def\O{\mathcal O}
\def\F{\mathcal F}

\def\({\left(}
\def\){\right)}
\def\<{\left\langle}
\def\>{\right\rangle}
\def\le{\leqslant}
\def\ge{\geqslant}

\def\1{{\mathbf 1}}

\def\d{{\partial}}
\def\eps{\varepsilon}
\def\l{\lambda}
\def\om{\omega}
\def\si{{\sigma}}

\DeclareMathOperator{\RE}{Re}
\DeclareMathOperator{\IM}{Im}
\DeclareMathOperator{\DIV}{div}

\numberwithin{equation}{section}

\begin{document}

\title[On NLS with nonlinear damping]{On nonlinear Schr\"odinger type
  equations with nonlinear damping}
\author[P. Antonelli]{Paolo Antonelli}
\address{Centro di Ricerca Matematica Ennio De Giorgi\\
Piazza dei Cavalieri, 3\\
56100 Pisa, Italy}
\email{paolo.antonelli@sns.it}

\author[R. Carles]{R\'emi Carles}
\address{CNRS \& Univ. Montpellier~2\\Math\'ematiques
\\CC~051\\34095 Montpellier\\ France}
\email{Remi.Carles@math.cnrs.fr}

\author[C. Sparber]{Christof Sparber}
\address{Department of Mathematics, Statistics, and Computer Science\\
University of Illinois, Chicago\\
851 South Morgan Street
Chicago, Illinois 60607, USA}
\email{sparber@uic.edu}

\begin{abstract} We consider equations of nonlinear Schr\"odinger type
  augmented by  
nonlinear damping terms. We show that nonlinear damping prevents
finite time blow-up in several situations, which we describe.  We
also prove that the presence of a quadratic confinement in all spatial
directions drives the solution of our model  
 to zero for large time. In the case without external potential we prove 
that the solution may not go to zero for large time due to
(non-trivial) scattering.  
\end{abstract}

\thanks{RC was supported by the French ANR projects
  SchEq (ANR-12-JS01-0005-01) and BECASIM
  (ANR-12-MONU-0007-04). CS acknowledges support by the NSF through
  grant no.   DMS-1161580.} 
\maketitle

\section{Introduction}
\label{sec:intro}
We consider, for $a>0$ and $\lambda\in\R$, the following class of damped nonlinear 
Schr\"odinger type equations (NLS):
\begin{equation}
  \label{eq:nlsp}
\left\{
  \begin{aligned}
      i\d_t u + \frac{1}{2}\Delta u & =V(x) u + \l
  |u|^{2\si_1} u -i a |u|^{2\si_2}u,\quad t\ge 0,\ x\in \R^d,\\
\quad u_{\mid t=0}&=u_0\in \Sigma, 
  \end{aligned}
\right.
\end{equation}
where $ \Sigma$ denotes the energy space associated to the harmonic
oscillator, i.e. 
\begin{equation*}
  \Sigma =\left\{ f\in H^1(\R^d),\ x\mapsto |x|f(x)\in L^2(\R^d)\right\}, 
\end{equation*}
equipped with the following norm:
\begin{equation*}
\| u\|_{\Sigma} := \| u \|_{L^2} + \| \nabla u \|_{L^2} + \| x u \|_{L^2}.
\end{equation*}
In the following, we allow $u_0$ to be arbitrarily large within $\Sigma$, i.e. 
we shall not be concerned with solutions corresponding to small initial data. 
We make the following standard assumption on the nonlinearities:
 \begin{equation*}
      0<\si_1,\si_2<\frac{2}{(d-2)_+},
    \end{equation*}
 where $(d-2)_+$ denotes the positive part. Thus, if $d\le 2$, we
 impose no size restriction on $\si_1, \si_2>0$. 
For $d\ge 3$ the above assumption ensures that both nonlinearities are
$H^1$-subcritical. 
The external potential $V$ is supposed to be harmonic (or zero),
\begin{equation*}
  V(x)= \frac{1}{2} \sum_{j=1}^d \om_j^2{x_j^2},\quad \om_j\ge 0.
\end{equation*}
As we shall indicate below, all of our results can be generalized to the case 
of potentials $V(x)\ge 0$, growing at most quadratically at infinity.
In the case $a=0$, it is well known that \eqref{eq:nlsp} is a
Hamiltonian equation (its mass and energy are conserved); see
e.g. \cite{CazCourant,Sulem}. In fact, the Hamiltonian counterpart of our model, 
i.e. \eqref{eq:nlsp} with $a\in i \R$ (yielding a nonlinear
Schr\"odinger equation with combined power-law nonlinearities) 
has been studied in \cite{TVZ07}.
In the present case $a>0$, the last term in \eqref{eq:nlsp} is dissipative, which is the reason why we consider
non-negative times only. Indeed, the dissipative nature of \eqref{eq:nlsp} can easily be seen from 
the fact that the local conservation law for the particle density $\rho= |u|^2$ is augmented as follows:
\begin{equation}\label{eq:conlaw}
\partial_t \rho + \text{div} J = - 2a \rho^{\si_2+1},
\end{equation}
where, as usual $J=\IM(\bar u \nabla u)$ denotes the current density. For $a>0$ the right hand side 
describes a nonlinear damping mechanism for the density.

Equations of the form \eqref{eq:nlsp} arise as phenomenological models in different 
areas of Physics. For example, in nonlinear optics, equation \eqref{eq:nlsp} with $V=0$
models the propagation of a laser pulse within an optical fiber ($d=1$) under the influence of 
additional multi-photon absorption processes, see, e.g., \cite{Bi01, Fi01}. Another application arises from
quantum mechanics, where NLS type models arise in the 
description of Bose-Einstein condensates in harmonic traps (which are experimentally required to 
produce these condensates). In this context the nonlinear damping 
is a model for the reduction of the condensate wave function through higher order 
particle-interactions, cf. \cite{Ad2002, BJM04}. 

From a mathematically point of view, NLS type equation with 
nonlinear damping terms have been studied in \cite{FiKl12} as a possibility to continue 
the solution of NLS beyond the point of finite-time blow-up (see also \cite{PaSuSu05} for an earlier study in this direction). 
In \cite{AnSp10}, global existence for the particular case $\sigma_1 = 1$ and $1< \sigma_2  \le 2$ in dimensions $d\le 3$ has been studied. 
Notice that in $d=3$, this allows to take into account an $H^1$-critical damping term (i.e. a quintic nonlinearity with $\sigma_2 =2$). 
In \cite{Da-p} the particular case of a mass critical nonlinearity $\sigma_1= 2/d$ and $V=0$ has been studied. In there, global in-time 
existence of solutions is established if $\sigma_2\ge 2/d$ and it is
claimed that finite time blow-up in the log-log regime occurs if
$\sigma_2 < 2/d$. A more complete understanding of the possibility of finite time blow-up remains an open problem, however 
(numerical simulations can be found in \cite{PaSuSu05, FiKl12}).\\

In the following, we shall develop a more systematic study of NLS type equations with nonlinear damping, 
generalizing the results mentioned above in several aspects:
\begin{itemize} 
\item We extend the results of global well-posedness to the case of general (energy-subcritical) nonlinearities.
\item We prove that in the case without external potential, the solution is asymptotically close to the solution of the free equation for $t\to +\infty$, i.e. 
we establish scattering for positive times.
\item We show that in the case where a quadratic external confinement
  is present in all spatial directions,  
the $L^2$ norm of solution vanishes asymptotically with 
a certain (not necessarily sharp) rate. (This result corrects a mistake in the proof of Corollary 2.5 in \cite{AnSp10}).
\item We compare the results above to the ones which can be obtained
  for $x\in M$, a compact manifold without boundaries. 
\end{itemize}

\begin{theorem}\label{theo:cauchyglobal}
  Let $d\ge 1$, $a>0$, $\om_1,\dots,\omega_d\ge 0$, and
  $u_0\in\Sigma$. Then, the Cauchy problem \eqref{eq:nlsp} has a unique solution $u\in
C(\R_+;\Sigma)$ in either one of the following cases:
 \begin{enumerate} 
\item $\l\ge 0$ (defocusing nonlinearity) and $0<\si_1,\si_2<2/(d-2)_+$;
\item $\l<0$ (focusing nonlinearity) with, either
\begin{enumerate}
\item $0<\si_1<2/d$ and $0<\si_2<2/(d-2)_+$, or
\item $2/d\le \si_1<\si_2<2/(d-2)_+$, or
\item $\si_1=\si_2=2/d$, or 
\item $2/d<\si_1=\si_2=:\si<2/(d-2)_+$ and
  $a\ge\min\(\si,\sqrt\si\) |\l|$. 
  \end{enumerate}
\end{enumerate}
If in addition  $\om_j>0$ for all $j$, then $u$ vanishes asymptotically as $t\to
+\infty$:
\begin{equation*}
  \|u(t)\|_{L^2(\R^d)}=\O\(t^{-\frac{2}{(d+2)\si_2}}\)\quad \text{as
  }t\to +\infty.  
\end{equation*}
If, however, $\omega_j = 0$ for all $j$, $u$ may behave
like the free evolution of a non-trivial asymptotic state $u_+\in
L^2(\R^d)$: $\|u(t)-e^{i\frac{t}{2}\Delta}u_+\|_{H^1(\R^d)}\to 0$ as
$t\to +\infty$.  
\end{theorem}

The paper is organized as follows: In the next section we establish the 
local (in-time) well-posedness of our model in the energy space. In Section \ref{sec:global}, 
we extend this to global in-time well-posedness using a modified energy functional. 
The long time behavior and the possible extinction of solutions is studied in 
Section \ref{sec:asymptotic}. Finally, we briefly discuss the case 
of compact manifolds in the appendix.

\section{Basic properties of the Cauchy problem}
\label{sec:cauchy}


In this section we shall show that \eqref{eq:nlsp} is locally well-posed for any $u_0\in \Sigma$ and we also establish a blow-up alternative. 

\subsection{Local well-posedness}
\label{sec:local}

We denote by 
$U(t)= e^{-itH}$, the Schr\"odinger group generated by $\dis H =-\frac{1}{2}\Delta +V$. We first recall the
standard Strichartz estimates (see e.g. \cite{CazCourant}).
\begin{definition}\label{def:adm}
 A pair $(q,r)$ is admissible if $2\le r
  <\frac{2d}{d-2}$ ($2\le r\le\infty$ if $d=1$, $2\le r<
  \infty$ if $d=2$)  
  and 
$$\frac{2}{q}=\delta(q):= d\left( \frac{1}{2}-\frac{1}{r}\right).$$
\end{definition}

\begin{proposition}[Strichartz estimates]\label{prop:strichartz}
  Let $T>0$.  There exists
$\eta>0$ such that the following holds:\\ 
$(1)$ For any admissible pair $(q,r)$, there exists $C_{q}$  such that
$$
\|U(\cdot) \varphi\|_{L^{q}([0,\eta];L^{r})} \le C_q 
\|\varphi \|_{L^2},\quad \forall \varphi\in L^2(\R^d).
$$
$(2)$ 
For $s\in \R$, denote
\begin{equation*}
  D_s(F)(t,x) = \int_{s}^t U(t-\tau)F(\tau,x)d\tau. 
\end{equation*}
For all admissible pairs $(q_1,r_1)$ and~$
    (q_2,r_2)$,  there exists $C=C_{q_1,q_2}$ independent of $s\in \R$
    such that  
\begin{equation}\label{eq:strichnl}
      \left\lVert D_s(F)
      \right\rVert_{L^{q_1}([s,s+\delta];L^{r_1})}\le C \left\lVert
      F\right\rVert_{L^{q'_2}\([s,s+\delta];L^{r'_2}\)},
\end{equation}
for all $F\in L^{p'_2}(I;L^{q'_2})$ and $0\le \delta\le \eta$.\\
$(3)$ In the case without potential, $V=0$, the above results remain true
with $\eta=\infty$. 
\end{proposition}

\begin{proposition}[Local existence]\label{prop:Sigmaloc2}
  Let $\l,a\in \R$, $\om_1,\dots,\omega_d\ge 0$, and    $\si_j>0$
  with $\si_j<2/(d-2)$ if 
  $d\ge 3$.  For  all
  $u_0\in \Sigma$,  
  there exists  $T$  and a
  unique solution $u$ of \eqref{eq:nlsp}, such that
  \begin{equation*}
    u,\nabla u,xu \in C\([0,T];L^2(\R^d)\)\cap
  \bigcap_{j=1,2}L^{\frac{4\si_j+4}{d\si_j}}\([0,T];L^{2\si_j+2}(\R^d)\).
  \end{equation*}
\end{proposition}

\begin{proof}
We present the main steps of the classical
  argument, which can be found for instance in
  \cite{CazCourant,Ginibre} in the 
  case $V=0$ (see also \cite{Ca11} in the presence of a
  potential). Duhamel's formulation for \eqref{eq:nlsp} reads 
  \begin{equation}\label{eq:duham}
    u(t)=U(t)u_0 -i\l \int_0^t U(t-\tau)\(|u|^{2\si_1}u\)(\tau)d\tau-a
    \int_0^t U(t-\tau)\(|u|^{2\si_2}u\)(\tau)d\tau. 
  \end{equation}
Denote the right hand side by $\Phi(u)(t)$. 
Proposition~\ref{prop:Sigmaloc2} follows from a fixed point argument
in a ball of the space
\begin{equation*}
  X_T = \left\{u\in C([0,T];\Sigma)\ ;\ u,xu,\nabla u \in
    \bigcap_{j=1,2}L^{\frac{4\si_j+4}{d\si_j}}_TL^{2\si_j+2} \right\},
\end{equation*}
where $L^q_TL^r$ stands for $L^q([0,T];L^r(\R^d))$.
For $j=1,2$, introduce the Lebesgue exponents 
\begin{equation}\label{eq:exponents}
  r_j=2\si_j+2\quad ;\quad q_j=\frac{4\si_j+4}{d\si_j}\quad ;\quad
  \theta_j=\frac{2\si_j(2\si_j+2)}{2-(d-2)\si_j}. 
\end{equation}
Then $(q_j,r_j)$ is admissible, and
\begin{equation*}
  \frac{1}{r_j'}=\frac{2\si_j}{r_j}+\frac{1}{r_j}\quad ;\quad
  \frac{1}{q_j'}=\frac{2\si_j}{\theta_j} +\frac{1}{q_j}. 
\end{equation*}
Proposition~\ref{prop:strichartz} and H\"older inequality yield, for $j=1,2$
\begin{align*}
  \|\Phi(u)\|_{L^{q_j}_TL^{r_j}\cap L^\infty_T
    L^2}&\le C \|u_0\|_{L^2} 
  +C \sum_{\ell=1,2}\left\lVert
    \lvert u\rvert^{2\si_\ell}u\right\rVert_{L^{q_\ell'}_T L^{r_\ell '}}\\
&\le C \|u_0\|_{L^2}
  +C \sum_{\ell=1,2}\|u\|_{L^{\theta_\ell}_T L^{r_\ell}}^{2\si_\ell
  } \|u\|_{L^{q_\ell}_T L^{r_\ell}}, 
\end{align*}
where $C$ is independent of $T\le \eta$. We note that if $\si_\ell\le
2/d$ for $\ell=1,2$, then $\theta_\ell \le q_\ell$, and
\begin{equation}\label{eq:strich_L2}
  \|\Phi(u)\|_{L^{q_j}_T L^{r_j}\cap L^\infty_T
    L^2}\le C \|u_0\|_{L^2}
  +C \sum_{\ell=1,2}T^{2\si_\ell
    (1/\theta_\ell-1/q_\ell)}\|u\|_{L^{q_\ell}_TL^{r_\ell}}^{2\si_\ell 
 +1 }.
\end{equation}
If $\si_\ell>2/d$ for $\ell=1$ or $2$, using Sobolev
embedding, 
\begin{equation*}
 \|\Phi(u)\|_{L^{q_j}_T L^{r_j}\cap L^\infty_T
    L^2} \le C \|u_0\|_{L^2}
  +C \sum_{\ell=1,2}T^{2\si_\ell/\theta_\ell
  }\|u\|_{L^\infty_T H^1}^{2\si_\ell} \|u\|_{L^{q_\ell}_T L^{r_\ell}}. 
\end{equation*}
We have
\begin{align*}
  \nabla \Phi(u)(t) &= U(t)\nabla u_0 -i\l \int_0^t
  U(t-\tau)\nabla \(|u|^{2\si_1}u\)(\tau)d\tau \\
&-a \int_0^t
  U(t-\tau)\nabla \(|u|^{2\si_2}u\)(\tau)d\tau
-i\int_0^t U(t-\tau)\(\Phi(u)(\tau)\nabla V(\tau)\)d\tau. 
\end{align*}
We estimate the second term of the right hand side as above, and, for
all admissible pairs $(q,r)$, the 
new term is estimated by
\begin{align*}
\left\| \int_0^t U(t-\tau)\(\Phi(u)(\tau)\nabla V(\tau)\)
d\tau\right\|_{L^{q}_TL^{r}} &\le
C \|\Phi(u)\nabla V\|_{L^1_T L^2}\\
& \le C T \( \|\Phi(u)\|_{L^\infty_T L^2}+\|x
\Phi(u)\|_{L^\infty_T L^2}\),
\end{align*}
where we have written an estimate which is valid in the more general
case where $V$ is at most quadratic ($\d^\alpha V\in L^\infty(\R^d)$
for $|\alpha|\ge 2$). 
Similarly, to estimate $x \Phi(u)$, a new term appears, which is
controlled by 
$$C T\|\nabla \Phi(u)\|_{L^\infty_TL^2}.$$
\smallbreak

Choosing $T$ sufficiently small, one can then prove that $\Phi$ maps
a suitable ball in $X_T$ into itself. Contraction for the norm
$\|\cdot\|_{L^{q_1}(I_T;L^{r_1})}$ is proved similarly, and one concludes by
remarking that $X_T$ equipped with this norm 
is complete. 
\end{proof}

\begin{remark}
  The above result can be extended to the case where the first
  assumption is replaced with $u\in L^\infty(\R_+;\F(H^s))$
  for some $s>0$, up to changing the application of H\"older
  inequality in the proof. \end{remark}

\begin{remark}[Energy-critical damping]
 When $d\ge 3$, the case $\si_2=2/(d-2)$
 could be considered, like in \cite{AnSp10} for the case $d=3$. This
  requires a different presentation in the proofs, which is the
 reason why this case is not studied here. 
\end{remark}

\begin{remark}[More general potentials]
  As suggested in the course of the proof,
  Proposition~\ref{prop:Sigmaloc2} remains valid if we assume more
  generally that $V(x)$ is smooth, and at most quadratic, i.e. $\d^\alpha V\in
  L^\infty(\R^d)$ for all $|\alpha|\ge 2$.
\end{remark}

\subsection{Basic \emph{a priori} estimates and blow-up alternative}
\label{sec:apriori}

In order to extend the obtained local in-time solution to arbitrary time intervals, we shall derive several 
{\it a priori} estimates. In a first step, we recall \eqref{eq:conlaw} to infer the following.

\begin{lemma}[Mass dissipation] \label{lem:L2}
The local in time solution $u(t)\in \Sigma$ satisfies
\begin{equation}  \label{eq:L2}
 \frac{d}{dt}\| u(t)\|_{L^2}^2 +
 2a\|u(t)\|_{L^{2\si_2+2}}^{2\si_2+2}=0,\quad \forall t\in [0,T].  
\end{equation}
As a consequence, we have that $u\in L^\infty([0,T]\times L^2(\R^d)) \cap  L^{2\si_2+2}([0,T]\times \R^d)$.
\end{lemma}

\begin{proof} We multiply \eqref{eq:nlsp} by $\bar u$ and
integrate with respect to $x\in \R^d$. Taking the real part, yields \eqref{eq:L2}, which consequently implies
 \[ \frac{d}{dt} \| u (t)\|^2_{L^2} \le -2a\|u(t)\|_{L^{2\si_2+2}}^{2\si_2+2}\le 0,\] and thus $\| u(t)\|_{L^2} \le \| u_0 \|_{L^2}$, $\forall t \in [0,T]$. 
In addition, we can integrate \eqref{eq:L2} with respect to $t$ to infer 
\begin{equation*}
2 a \int_0^T \|u(t)\|_{L^{2\si_2+2}}^{2\si_2+2} \, dt =  \| u_0\|_{L^2}^2 - \| u(T)\|_{L^2}^2 \le \| u_0\|_{L^2}^2,
\end{equation*}
and thus $u \in L^{2\si_2+2}([0,T]\times \R^d)$. 
\end{proof}

\begin{remark}[Non-existence of steady states]
An immediate consequence of \eqref{eq:L2} is the non-existence of
non-trivial steady states.  
In the Hamiltonian case ($a=0$), they are found by inserting the ansatz $u(t,x) = \psi(x)e^{i\mu t}$ with $\mu \in \R$ into \eqref{eq:nlsp}
and study the resulting elliptic equation for $\psi$. In our case, \eqref{eq:L2} together with the fact that for stationary states 
$|u(t,x)|^2 = |\psi(x)|^2$, immediately implies that $\psi=0$.
\end{remark}

Proposition \ref{prop:Sigmaloc2} and Lemma \ref{lem:L2} allow us to infer the following blow-up alternative.

\begin{corollary}[Blow-up alternative]\label{cor:alternative}
Let $\l\in \R$, $a,\om_1,\dots,\omega_d\ge 0$,    $\si_j>0$
  with $\si_j<2/(d-2)$ if 
  $d\ge 3$,   and $u_0\in \Sigma$. Either the solution to 
  \eqref{eq:nlsp} exists for all $t\ge 0$, i.e. 
  \begin{equation}\label{eq:global}
    u,\nabla u,xu \in C\(\R_+;L^2(\R^d)\)\cap\bigcap_{j=1,2}
    L^{\frac{4\si_j+4}{d\si_j}}_{\rm 
      loc}\(\R_+;L^{2\si_j+2}(\R^d)\), 
  \end{equation}
or there exists $T>0$, such that 
\begin{equation*}
  \|\nabla u(t)\|_{L^2}
\mathop{\longrightarrow}\limits_{t{\mathop{\rightarrow}\limits_<}  T} 
+\infty.  
\end{equation*}
In the case $\si_1=\si_2=\frac{2}{d}$, or in the fully
mass-subcritical case $\si_1,\si_2<2/d$, the solution is global, that
is, \eqref{eq:global} is satisfied. 
\end{corollary}
\begin{proof}
Let $M>0$. 
  Lemma~\ref{lem:L2} shows that 
  $t\mapsto\|u(t)\|_{L^2}^2$ is non-increasing function. Thus, the only obstruction to
  well-posedness on  $[0,M]$ is the existence of a
  time $0<T<M$ such that 
\begin{equation*}
  \|x u(t)\|_{L^2}+\|\nabla u(t)\|_{L^2}
\mathop{\longrightarrow}\limits_{t{\mathop{\rightarrow}\limits_<}  T} 
 +\infty.  
\end{equation*}
As long as $u\in
C([0,t];\Sigma)$, we have 
\begin{align*}
   \frac{d}{dt}\int_{\R^d}x_j^2 \lvert u(t,x)\rvert^2dx &= 2\RE \int
   x_j^2 \overline u(t,x) \d_t u (t,x)dx = 2\IM \int
   x_j^2 \overline u(t,x) i\d_t u(t,x)dx \\
& = -\IM \int
   x_j^2 \overline u(t,x) \Delta u(t,x) - a\int x_j^2
   |u(t,x)|^{2\si_2+2}dx\\
& = 2\IM \int
   x_j \overline u(t,x) \d_j u(t,x) - a\int x_j^2
   |u(t,x)|^{2\si_2+2}dx\\
&\le 2 \|x u(t)\|_{L^2}\|\nabla u(t)\|_{L^2},
\end{align*}
where we have used Cauchy--Schwarz inequality and the assumption $a\ge
0$. 
Suppose  $u\in L^\infty ([0,T];H^1)$. Then the above estimate and
Gronwall's lemma show that $x u\in L^\infty 
([0,T];L^2)$,
hence a contradiction. Hence the first part of the corollary. The
second part follows from the  standard criterion for $L^2$-critical
problems (see e.g. \cite{CazCourant}): if the solution does not satisfy
\eqref{eq:global}, then there exists $T>0$ such that
\begin{equation*}
  \int_0^T \|u(t)\|_{L^{2+4/d}}^{2+4/d}dt=\infty.
\end{equation*}
This is also a direct consequence of the proof of
Proposition~\ref{prop:Sigmaloc2} (if $\si_1=\si_2=2/d$, then
$\theta_j=q_j$). Lemma~\ref{lem:L2} rules out this possibility; 
this point has already been noticed in \cite{Da-p}. 

Finally, if $\si_1,\si_2<2/d$, the standard
argument given initially in \cite{TsutsumiL2}, following again from
Proposition~\ref{prop:Sigmaloc2}, shows that \eqref{eq:global} follows
if $u\in L^\infty(\R_+;L^2(\R^d))$, which in turn is a direct
consequence of Lemma~\ref{lem:L2}. 
\end{proof}

\section{Global well-posedness}
\label{sec:global}

The {\it a-priori} bounds obtained in Section \ref{sec:apriori} are
not sufficient to infer global well-posedness for $\si_j\ge 2/d$
(unless $\si_1=\si_2=2/d$). 
In order to obtain further {\it a-priori} estimates, one possible approach would be to follow \cite{Da-p}, where 
the author studies the time-evolution of $\| \nabla u(t) \|_{L^2}$ and shows that for $\si_2>\si_1=2/d$ one 
can obtain a bound of the form
\[
\| \nabla u(t) \|_{L^2} \le \| \nabla u_0 \|_{L^2} \, e^{C t},
\]
for some $C>0$ depending on the involved parameters $a, \l, \sigma_2$. Indeed, an analogous result could also be obtained in our, more general, situation.
However, we shall rather follow the approach of \cite{AnSp10} based on a modified energy functional $E(t)$ which will allow us to infer (under certain conditions) 
uniform in-time bounds on different quantities involving $u(t)$.

\subsection{Bounds on a modified energy functional}

In the following, we denote 
\begin{equation}\label{eq:energy}
\begin{split}
E(t)= & \ \frac{1}{2} \| \nabla u(t)\|_{L^2}^2 + \int_{\R^d} V(x) |u(t,x)|^2 \, dx\\ 
& \ + \frac{\lambda}{\sigma_1+1}  \int_{\R^d}  |u(t,x)|^{2\sigma_1+2} dx + \kappa \int_{\R^d} |u(t,x)|^{2\sigma_2 +2} \, dx,
\end{split}
\end{equation}
for some $\kappa>0$ (to be made precise below). Clearly, $E$ is
well defined on $[0,T]$ for any $u\in C([0,T];\Sigma)$, by Sobolev embedding.  
Even though this energy functional is not conserved, we shall prove
that it is uniformly bounded in time, provided that some assumptions on the
involved parameters hold true. 

\begin{proposition} [Energy bound] \label{prop:enbound}
Let $0<\kappa<\frac{a}{\si_2^2+\si_2}$ and assume that
\begin{enumerate}
\item Either $\lambda \ge 0$, 
\item Or $\l <0$ and $\sigma_2>\si_1$.
\end{enumerate}
Then there exists a $C=C(\|u_0\|_{L^2})\ge 0$ such that:
\begin{equation*}
E(t) \le E(0) +  C(\|u_0\|_{L^2}),\quad \forall t\in [0,T],  
\end{equation*}
where $T>0$ is the existence time obtained in Proposition \ref{prop:Sigmaloc2}.
\end{proposition}

\begin{proof}
We first assume that $u(t)$ is sufficiently regular an decaying so that all of the following formal manipulations can be carried out. 
Once the final result is established, a standard density argument allows to conclude that it also holds for $u\in C([0,T];\Sigma)$.

We compute the time-derivative of the energy functional \eqref{eq:energy}, using equation \eqref{eq:nlsp}, which yields:
\begin{align}\label{eq:time_der_en}
\frac{d}{dt} E(t) = &\ a  \int_{\R^d} |u|^{2\si_2}\RE(  u \Delta \bar u) \, dx - \kappa (\si_2 + 1) \int_{\R^d} |u|^{2\si_2} \IM (\bar u \Delta u) \, dx \\
\notag& \, - 2a \int_{\R^d} V(x) |u|^{2\si_2 +2} \, dx - 2a \lambda \int_{\R^d} |u|^{2\si_2 +2\si_1 +2} \, dx \\
\notag& \, - 2a \kappa (\si_2+1)  \int_{\R^d} |u|^{4\si_2 +2} \, dx.
\end{align}
Consider the first term on the right hand side; since
$\Delta|u|^2=2\RE(\bar u\Delta u)+2|\nabla u|^2$, notice it can be
rewritten, using integration by parts, 
\begin{equation*}
\int_{\R^d}|u|^{2\sigma_2}\RE(\bar u\Delta u)dx=-\int_{\R^d}|u|^{2\sigma_2}|\nabla u|^2dx
-2\sigma_2\int_{\R^d}|u|^{2\sigma_2}\big|\nabla|u|\big|^2dx.
\end{equation*}
Now, we rewrite the second term on the right hand side:
\begin{equation*}
\int_{\R^d}|u|^{2\sigma_2}\IM(\bar u\Delta u)dx=
\int_{\R^d}|u|^{2\sigma_2}\DIV\big(\IM(\bar u\nabla u)\big)dx=
-\int_{\R^d}\nabla|u|^{2\sigma_2}\cdot\IM(\bar u\nabla u)dx.
\end{equation*}
Here we use the polar factorisation introduced in \cite{AM} (see also \cite{CDS}), to show the above integral equals
\begin{equation}\label{eq:second_term}
-2\sigma_2\int_{\R^d}|u|^{2\sigma_2}\RE(\bar\phi\nabla u)\cdot\IM(\bar\phi\nabla u)dx,
\end{equation}
where $\phi$ is the polar factor related to $u$, 
\begin{equation*}
  \phi(t,x):=
\left\{
  \begin{aligned}
   & |u(t,x)|^{-1}u(t,x) & \quad\text{ if }u(t,x)\not =0,\\
&0& \quad\text{ if }u(t,x) =0.
  \end{aligned}
\right.
\end{equation*}
This indeed can first be proved by replacing $\phi$ with 
\begin{equation*}
  \phi^\eps(t,x) = \frac{u(t,x)}{\sqrt{|u(t,x)|^2+\eps^2}},
\end{equation*}
and then passing to the limit $\eps \to 0$ in $H^1$, as in \cite{AM},
\cite{CDS}. Let us rewrite \eqref{eq:second_term} as 
\begin{multline*}
\sigma_2\int_{\R^d} |u|^{2\sigma_2}\big|\RE(\bar\phi\nabla u)-\IM(\bar\phi\nabla u)\big|^2
-|u|^{2\sigma_2}\bigg(|\RE(\bar\phi\nabla u)|^2+|\IM(\bar\phi\nabla u)|^2\bigg)dx\\
=\sigma_2\int_{\R^d}|u|^{2\sigma_2}|\RE(\bar\phi\nabla u)-\IM(\bar\phi\nabla u)|^2dx
-\sigma_2\int_{\R^d}|u|^{2\sigma_2}|\nabla u|^2dx.
\end{multline*}
The last equality follows from the identity
\begin{equation*}
|\nabla u|^2=|\RE(\bar\phi\nabla u)|^2+|\IM(\bar\phi\nabla u)|^2,\qquad\textrm{a.e. in}\;\R^d,
\end{equation*}
see formulas (30) in \cite{AM} and (5.15) in \cite{CDS}.
Hence, by resuming the second term on the right hand side of \eqref{eq:time_der_en} is equal to
\begin{equation*}
-\kappa(\sigma_2^2+\sigma_2)\int_{\R^d}|u|^{2\sigma_2}|\RE(\bar\phi\nabla u)-\IM(\bar\phi\nabla u)|^2dx
+\kappa(\sigma_2^2+\sigma_2)\int_{\R^d}|u|^{2\sigma_2}|\nabla u|^2dx.
\end{equation*}
In summary this yields:
\begin{align*}
&\, \frac{d}{dt} E(t) = \, -\(a-\kappa(\si_2^2+\si_2)\)
\int_{\R^d}|u|^{2\sigma_2}|\nabla u|^2dx -2a\si_2 \int_{\R^d}
|u|^{2\si_2}\left\lvert \nabla |u|\right\rvert^2 \\
&-\kappa(\si_2^2+\si_2)\int_{\R^d}|u|^{2\sigma_2}|\RE(\bar\phi\nabla
u)-\IM(\bar\phi\nabla u)|^2dx 
 - 2a \int_{\R^d} V(x) |u|^{2\si_2 +2} \, dx \\
&- 2a \lambda \int_{\R^d}
 |u|^{2\si_2 +2\si_1 +2} \, dx  - 2a \kappa (\si_2+1)  \int_{\R^d}
 |u|^{4\si_2 +2} \, dx. 
\end{align*}
Under the assumption $0<\kappa<\frac{a}{\si_2^2+\si_2}$, and if
$\lambda \ge 0$ (defocusing case), all the terms on the right hand side are
non-positive, and we infer that $E$ is a non-increasing function:
$E(t)\le E(0)$, $\forall t\in [0,T]$.  

If $\l <0$ (focusing case), however, the term involving
the $L^{2\si_1+2\si_2+2}$ norm is positive. But since $\si_2>\si_1$
by assumption, we  
can interpolate this term using:
\[
\| u \|_{L^{2\si_1+2\si_2+2}} \le \| u \|^{\theta}_{L^{4\si_2+2}} \,
\| u \|^{1-\theta} _{L^{2\si_2+2}} \, ,  
\]
with 
\[
\frac{1}{\si_1+\si_2+1} = \frac{\theta}{2\si_2+1} +
\frac{1-\theta}{\si_2+1}: \quad \theta =
\frac{\si_1(2\si_2+1)}{\si_2(\si_1+\si_2+1)}\in (0,1).
\]
Denoting $\gamma= \si_1/\si_2$, this implies that
\[
\| u \|^{2\si_1+2\si_2+2}_{L^{2\si_1+2\si_2+2}} \le \| u \|^{\gamma(4\si_2+2)}_{L^{4\si_2+2}} \, \| u \|^{(1-\gamma)(2\si_2+2)} _{L^{2\si_2+2}} \, 
\le \eps  \| u \|^{4\si_2+2}_{L^{4\si_2+2}}  + \frac{1}{\eps^{\gamma/(1-\gamma)}}  \| u \|^{2\si_2+2} _{L^{2\si_2+2}}, 
\]
where we have used Young inequality. In summary this yields
\begin{align*}
\frac{d}{dt} E(t) \le &  - 2a \big ( \kappa (\si_2+1) -| \lambda| \eps \big)  \int_{\R^d} |u|^{4\si_2 +2} \, dx + \frac{2a |\lambda| }{\eps^{\gamma/(1-\gamma)}} \int_{\R^d} |u|^{2\si_2 +2} \, dx.
\end{align*}
For $0< \eps \ll 1$, the coefficient in front of the first term
is negative, which allows to conclude, after an integration with
respect to time, that: 
\[
\forall t\in [0,T]: \ E(t)  \le E(0) + C  \int_0 ^T \int_{\R^d} |u|^{2\si_2 +2} \, dx \, dt < E(0) + C(\|u_0\|_{L^2}),
\]
since we have $u \in L^{2\si_2+2}([0,T] \times \R^d)$ from Lemma~\ref{lem:L2}.
\end{proof}


We are now in the position to prove global well-posedness of \eqref{eq:nlsp} 
under various conditions on the parameters. For the sake of a simpler presentation, 
we shall treat the case $\sigma_1 = \sigma_2$ separately, see Section~\ref{sec:equal} below.

\subsection{The case $\sigma_1 \not = \sigma_2$}\label{sec:nonequal}

In the following we shall consider two different powers and impose the following assumption.

\begin{hyp}[Nonlinearity]\label{hyp:NL} Let $\sigma_1 \not = \si_2$ and, in addition:
  \begin{enumerate}
  \item Defocusing case. If $\l\ge 0$, we assume
    \begin{equation*}
      0<\si_1,\si_2<\frac{2}{d-2}\quad (0<\si_1,\si_2\text{ if }d\le
      2).
    \end{equation*}
\item Focusing case. If $\l<0$, we assume
  \begin{itemize}
  \item Either, $\dis 0<\si_1<\frac{2}{d}$ and $\dis
    0<\si_2<\frac{2}{(d-2)_+}$,  
 \item Or  $\dis \frac{2}{d}\le \si_1<\si_2<\frac{2}{(d-2)_+}$.
  \end{itemize}
  \end{enumerate}
\end{hyp}

\begin{remark}
 Assumption~\ref{hyp:NL} can be understood as follows: If without
 damping ($a=0$), the solution of \eqref{eq:nlsp} is global in time,
 then the strength of the nonlinear damping plays no role. On the
 other hand, if finite time blow-up may occur in the Hamiltonian case
 (i.e., $\l<0$ and $\si_1\ge 2/d$), then the damping is assumed to be
 stronger than the attractive interaction, in terms of the power
 $\si_2$.  
\end{remark}

\begin{theorem}[Global existence I]\label{theo:cauchy}
Let $d\ge 1$,  $a>0$, and $\om_1,\dots,\omega_d\ge 0$. Under
Assumption~\ref{hyp:NL}, for any 
$u_0\in\Sigma$, \eqref{eq:nlsp} has a unique solution
\begin{equation*}
  u \in C(\R_+;\Sigma)\cap L^\infty(\R_+;H^1)\cap L^{2\si_2+2}(\R_+\times \R^d)\cap
  \bigcap_{j=1,2}L^{\frac{4\si_j+4}{d\si_j}}_{\rm 
      loc}\(\R_+;L^{2\si_j+2}(\R^d)\).
\end{equation*}
Moreover, if $\om_j>0$ for all $j\in \{1,\dots,d\}$, then we also have
$u \in L^\infty(\R_+;\Sigma)$.
\end{theorem}

\begin{proof} In the defocusing situation $\l>0$, Proposition \ref{prop:enbound} immediately implies a uniform (in-time) bound on $\| \nabla u(t)\|^2_{L^2} \lesssim E(0)$, since
$E(t)$ is the sum of four non-negative terms. In view of the blow-up alternative, stated in Corollary~\ref{cor:alternative}, we thus infer global well-posedness.

For the focusing situation $\l<0$, we note that the mass-subcritical case $\si_1,\si_2< 2/d$, has already been dealt with in Corollary~\ref{cor:alternative} (including 
$\si_1=\si_2=2/d)$. The moment $\si_1\ge 2/d$, we require $\si_2>\si_1$ which allows us to interpolate the potential energy of the attractive nonlinearity 
similarly to the calculation above, namely:
\begin{align*}
 \| u(t) \|^{2\si_1+2}_{L^{2\si_1+2}}\le  \| u(t) \|^{2\beta}_{L^2}   \, \| u(t) \|^{(1-\beta) (2\si_2+2)}_{L^{2\si_2+2}} \le {\eps}^{(\beta-1)/\beta} \| u(t) \|_{L^2}^2 + \eps \| u(t) \|^{2\si_2+2}_{L^{2\si_2+2}},
 \end{align*} 
for any $\eps > 0$ and an appropriately chosen $0<\beta<1$. Now fix $\eps = 2 \kappa (\si_1+1)/|\lambda|$ to obtain
\begin{align*}
\| \nabla u(t)\|^2_{L^2} \le 2 E(t) + C  \| u(t) \|^{2}_{L^2} ,
\end{align*} 
for some $C= C(\lambda, \kappa, \si_j)>0$. In view of Corollary~\ref{cor:alternative} and Proposition \ref{prop:enbound} this yields $\| \nabla u(t)\|^2_{L^2}\lesssim E(0) + \|u_0\|^2_{L^2}$ and we are done.
\end{proof}

\subsection{The case $\sigma_1 = \sigma_2$}\label{sec:equal}

In the case where $\sigma_1 = \sigma_2\equiv \si$, \eqref{eq:nlsp} simplifies to
\begin{equation}  \label{eq:nlssi}
      i\d_t u + \frac{1}{2}\Delta u  =V(x) u + (\l - i a)
  |u|^{2\si} u ,\quad  u_{\mid t=0}=u_0\in \Sigma. 
\end{equation}
Formally, this model can be considered as the diffusionless limit of
the (generalized) complex Ginzburg-Landau equation. 
Concerning global existence of \eqref{eq:nlssi}, we shall only be interested in the $L^2$-supercritical situation ($\si > 2/d$), in view of the last assertion 
in Corollary~\ref{cor:alternative}.

\begin{theorem}[Global existence II]\label{theo:cauchysi}
Let $d\ge 1$,  $a>0$, $\si > 2/d$, and $\om_1,\dots,\omega_d\ge 0$. Assume that
\begin{enumerate}
\item Either $\lambda \ge 0$,
\item Or $\lambda <0$ and $a\ge\min\(\sigma,\sqrt\sigma\) |\l|$.
\end{enumerate} 
Then, for all $u_0\in\Sigma$, \eqref{eq:nlssi} has a unique solution
\begin{equation*}
  u \in C(\R_+;\Sigma)\cap L^\infty(\R_+;H^1)\cap
  L^{2\si+2}(\R_+\times \R^d)\cap L^{\frac{4\si+4}{d\si}}_{\rm  
      loc}\(\R_+;L^{2\si+2}(\R^d)\).
\end{equation*}
Moreover, if $\om_j>0$ for all $j\in \{1,\dots,d\}$, then we also have
$u \in L^\infty(\R_+;\Sigma)$.
\end{theorem}
 
\begin{proof}
The defocusing case $\lambda \ge 0$, can be treated as in the proof of Theorem \ref{theo:cauchy} above. 
It therefore remains to consider the case $\l<0$ and $\sigma >
2/d$. Note that we cannot invoke
Proposition~\ref{prop:enbound},  
which requires $\sigma_1<\sigma_2$. Instead (following an idea in
\cite{AnSp10}), we  consider the linear energy functional: 
\[
E_{\rm lin} (t) = \frac{1}{2} \| \nabla u(t)\|_{L^2}^2 + \int_{\R^d} V(x) |u(t,x)|^2 \, dx.
\] 
Obviously, it suffices to prove that $E_{\rm lin} (t)$ is uniformly
bounded in time.  
Differentiating $E_{\rm lin}(t)$ and using Equation~\eqref{eq:nlssi} yields
\begin{align}\label{eq:lin_en_der}
\frac{d}{d t} \, E_{\rm lin}(t)= & \,
 a \int_{\R^d}| u |^{2\si} \RE(\bar{ u }\Delta u )\, d x  +|\lambda| \int_{\R^d}| u |^{2\si} \IM( u \Delta\bar{ u })\, d x \\
 &\, - 2 a \int_{\R^d} V(x) |u|^{2\si+2} \, dx.\notag
\end{align}
Using the same arguments as in the proof of
Proposition~\ref{prop:enbound}, we infer 
\begin{align*}
\frac{d}{d t} \,  E_{\rm lin}(t)\le & -\(a-|\l|\si\)
\int_{\R^d}|u|^{2\sigma}|\nabla u|^2dx -2a\si \int_{\R^d}
|u|^{2\si}\left\lvert \nabla |u|\right\rvert^2 \\
&-|\l|\si \int_{\R^d}|u|^{2\sigma}|\RE(\bar\phi\nabla
u)-\IM(\bar\phi\nabla u)|^2dx 
 - 2a \int_{\R^d} V(x) |u|^{2\si +2} \, dx . 
\end{align*}
Thus, if $a\ge \sigma|\lambda|$, we  obtain $E_{\rm lin}(t) \le E_{\rm
  lin}(0) < \infty$, for all $t\in [0,T]$.
  On the other hand, from \eqref{eq:lin_en_der} we have
  \begin{equation*}
  \begin{aligned}
  \frac{d}{d t}\, E_{\rm lin}(t)=&\,
  -a\sigma\int_{\R^d}|u|^{2\sigma}|\nabla|u||^2\,dx
  -a\int_{\R^d}|u|^{2\sigma}|\nabla\ u|^2\,dx\\
  &+2\sigma|\lambda|\int_{\R^d}|u|^{2\sigma}\nabla|u|\cdot\IM(\bar{\phi}\nabla u)\,dx
  - 2 a \int_{\R^d} V(x) |u|^{2\si+2} \, dx.
  \end{aligned}
  \end{equation*}
  By using Cauchy-Schwarz and then Young inequality, we see that the
  third term in the right hand side is bounded by 
  \begin{equation*}
  a\sigma\int|u|^{2\sigma}|\nabla|u||^2\,dx+\frac{|\lambda|^2\sigma}{a}\int|u|^{2\sigma}|\nabla u|^2\,dx,
  \end{equation*}
  hence if $|\lambda|\sqrt\sigma\le a$, then again we have $E_{\rm lin}(t) \le E_{\rm lin}(0) < \infty$, for all $t\in [0,T]$.
  This establishes global well-posedness of \eqref{eq:nlssi}, in the
second case. 
\end{proof}

\begin{remark}
The second case of Theorem \ref{theo:cauchysi} shows that if the damping is of the same strength (in terms of its power) 
as an attractive interaction nonlinearity, the relative size of the respective coefficients starts to play a role. Note that the larger the 
dimension $d\ge 1$, the smaller $a$ can be
chosen to ensure global existence. We remark that the numerical 
experiments presented in \cite{FiKl12} always consider $a \ll |\l|$ and thus they do not show the aforementioned possibility for global existence.
\end{remark}

At this point, we have proved global
existence in all the cases listed in
Theorem~\ref{theo:cauchyglobal}. We now turn to the large time
behavior, in cases where the solution is defined for all $t\ge 0$.

\section{Large time behavior}
\label{sec:asymptotic}

In view of the dissipation equation \eqref{eq:conlaw}, the damping is
expected to have a significant influence  
on the long time behavior of solutions to \eqref{eq:nlsp}.
Indeed, if we consider, for a moment,  the case of $x$-independent solutions, then 
\eqref{eq:conlaw} simplifies to 
the following ordinary differential equation
\begin{equation*}
\d_t \rho = - 2a \rho^{\sigma_2+1} , \quad \rho_{\mid t=0}=\rho_0:=|u_0|^2 ,
\end{equation*}
the solution of which is given by
\begin{equation}\label{eq:odesol}
\rho(t) = \frac{\rho_0}{(1+2 a t \rho_0^{\sigma_2}  )^{1/\sigma_{2}}}.
\end{equation}
Thus, one might expect the solution to vanish like $\|u(t)\|_{L^2}^2=
\O(t^{-1/\sigma_2})$, as $t\to +\infty$.  
We shall see below, however, that this rather naive argument does
not yield the correct  
long time behavior of $u$ in the case of where $V=0$. The idea is that
the dispersion due to the Laplacian may prevent the damping from
taking the wave function $u$ to zero. 

We mention in passing that the limiting case $\si_2=0$ corresponds to
an exponential decay rate, whose PDE analogue was studied in
\cite{OhTo09}, and the case $\si_2<0$ leads to finite time extinction
(see \cite{CaGa11} for the analogue regarding Schr\"odinger equation on
compact manifolds).

\subsection{Asymptotic extinction with full confinement}
\label{sec:extinct}
 In this subsection, we consider the case of a fully confining
 potential, in the sense 
 that we suppose $\om_j>0$ for \emph{all} $j$.
We start with a simple estimate, which can be viewed as a dual version
of Nash inequality.
\begin{lemma}[Localization] \label{lem:dualnash}
  Let $d\ge 1$. For all $p\ge 2$, there exists $C$
  such that for all $f\in \Sch(\R^d)$, 
\begin{equation*}
  \|f\|_{L^2(\R^d)}\le C
  \|f\|_{L^p(\R^d)}^{\theta}\|xf\|_{L^2(\R^d)}^{1-\theta},\quad \theta
  =  \frac{1}{1+d(1/2-1/p)}\in (0,1].
\end{equation*}
\end{lemma}
\begin{proof}
  For  $R>0$, write
  \begin{align*}
    \|f\|_{L^2(\R^d)} &\le \|f\|_{L^2(|x|<R)}+\|f\|_{L^2(|x|>R)}\\
&\le
    C_d R^{d(1/2-1/p)}\|f\|_{L^p(\R^d)} + \frac{1}{R}\|xf\|_{L^2(\R^d)}, 
  \end{align*}
where we have applied H\"older's inequality for the first term, and
simply multiplied and divided by $|x|$ for the second term. 
Both terms on the right hand side have the same order of magnitude if
$R^{1+d(1/2-1/p)}= \|xf\|_{L^2}/\|f\|_{L^p}$. Using this value of $R$
yields the result.
\end{proof}

Recall that in the case with confining potential ($\om_j>0$ for all
$j\in \{1,\dots,d\}$), the solution satisfies  
$u \in L^\infty(\R_+;\Sigma)$. This can be used to obtain an estimate
for the time-decay of the solution. 

\begin{proposition}[Asymptotic extinction I]\label{prop:asymextinct}
Let $d\ge 1$, $a>0$, $\om_1,\dots,\omega_d>0$, and
  $u_0\in\Sigma$. In either of the cases mentioned in
  Theorem~\ref{theo:cauchyglobal}, the solution to \eqref{eq:nlsp}
  satisfies $u\in L^\infty(\R_+;\Sigma)$ and there exists $C>0$ such that
\begin{equation*}
  \|u(t)\|_{L^2}^2\le Ct^{-\frac{2}{(d+2)\si_2}},\quad
  \forall t\ge 1. 
 \end{equation*}
\end{proposition}
\begin{proof}
In view of Lemma~\ref{lem:dualnash} with $p=2\si_2+2$, we have, since
$u\in L^\infty(\R_+;\Sigma)$ by assumption,
\begin{equation*}
  \|u(t)\|_{L^2}\le C \|u(t)\|_{L^{2\si_2+2}}^{\theta},\quad \theta =
  \frac{2\si_2+2}{(d+2)\si_2+2}. 
\end{equation*}
Along with \eqref{eq:L2}, this  yields
\begin{equation*}
  \frac{d}{dt}\|u(t)\|_{L^2}^2 + 2a C
  \|u(t)\|_{L^2}^{\frac{2\si_2+2}{\theta}}\le 0.
\end{equation*}
Therefore, $y(t)=
\|u(t)\|_{L^2}^2$ satisfies a differential inequality of the form 
\begin{equation*}
  \dot y(t) + C y(t)^p\le 0,\quad p =1+\frac{d+2}{2}\si_2.
\end{equation*}
The comparison with the ordinary differential equation yields $y(t) =
\O(t^{-1/(p-1)})$ for $t\ge 1$, since $p>1$, hence the result. 
\end{proof}



\subsection{Non-vanishing solutions}
\label{sec:scatt}
In the following, we assume that there is no confining potential, i.e.  $V=0$ in
\eqref{eq:nlsp}. Then the solution must not be expected to vanish as
$t\to +\infty$, as shown by the following result.
\begin{proposition}[Scattering without potential]\label{prop:scatt}
  Let $d\ge 1$, $\lambda,a \in \R$ and $2/d\le \si_1,\si_2\le 2/(d-2)$. Let
  $R>0$. There exists $T$ depending only 
  on $d$, $|\lambda|$, $|a|$ and $R$ such that for all $u_+\in H^1(\R^d)$
  with $\|u_+\|_{H^1}\le R$, there exists a solution $u\in
  C([T,\infty);H^1)$ to 
  \begin{equation*}
    i\d_t u + \frac{1}{2}\Delta u  = \l
  |u|^{2\si_1} u -i a |u|^{2\si_2}u,
  \end{equation*}
such that 
\begin{equation*}
\lim_{t\to +\infty} \dis \|u(t)-e^{i\frac{t}{2}\Delta}u_+\|_{H^1(\R^d)} = 0.
\end{equation*}
\end{proposition}

\begin{remark}
  By working with asymptotic states $u_+\in \Sigma$ instead of merely
  $H^1(\R^d)$, and using the operator $x+it\nabla$, the above result
  can be extended to the case where the lower bounds on $\si_1,\si_2$
  are relaxed (see for instance \cite{CazCourant}),
  \begin{equation*}
    \si_1,\si_2>1 \text{ if }d=1,\quad \si_1,\si_2>\frac{2}{d+2}
    \text{ if }d\ge 2. 
  \end{equation*}
\end{remark}

\begin{proof}[Sketch of the proof]
  This result follows from \cite{GiVe85b}
 (see also 
  \cite[Proposition~3.1]{Ginibre} for a simplification), and stems
  from a fixed point  
  argument applied to Duhamel's formulation of the above problem:
  \begin{equation*}
    u(t) = U(t)u_+ +i\lambda \int_t^\infty
    U(t-s)\(|u|^{2\si_1}u\)(s)ds +a \int_t^\infty
    U(t-s)\(|u|^{2\si_2}u\)(s)ds. 
  \end{equation*}
Denote by $\Phi(u)(t)$ the right hand side, and $u_{\rm free}(t) =
e^{i\frac{t}{2}\Delta}u_+$.  Let $r_j=2\si_j+2$ and $q_j>2$ be such
that $(q_j,2\si_j+2)$ 
is admissible, $j=1,2$, like in \eqref{eq:exponents}. 
With the notation $L^\beta_TY=L^\beta([T,\infty);Y)$, we
  introduce:  
  \begin{align*}
    X_T:=\Big\{ u\in C([T,\infty);H^1)\ ;\ &\left\|
    u\right\|_{L^{q_j}_TW^{1,2\si_j+2}} \le 2 C_{q_j}\|u_+\|_{H^1},\\
 \left\|  u\right\|_{L^\infty_TH^1} \le 2 \|u_+\|_{H^1}\, ,\quad
& \left\|  u\right\|_{L^{q_j}_T L^{2\si_j+2}} \le 2 \left\|
    u_{\rm free}\right\|_{L^{q_j}_T L^{2\si_j+2}},\ j=1,2 \Big\},
  \end{align*}
where $C_{q_j}$ is given by Proposition~\ref{prop:strichartz}. Resume
the numerology of the proof of Proposition~\ref{prop:Sigmaloc2}: we have
\begin{equation*}
  \frac{1}{r_j'}= \frac{1}{r_j}+\frac{2\si_j}{r_j},\quad
\frac{1}{q_j'}= \frac{1}{q_j}+\frac{2\si}{\theta_j},
\end{equation*}
where $q_j\le \theta_j<\infty$ since $2/d\le
\si_j<2/(d-2)$. For $u\in X_T$, Strichartz estimates and H\"older
inequality yield:
\begin{align*}
  \left\| \Phi(u)\right\|_{L^{q_j}_T W^{1,2\si_j +2}} &\le
  C_{q_j}\|u_+\|_{H^1} + C\sum_{\ell=1,2}\( \left\|
    |u|^{2\si_\ell}u\right\|_{L^{q_\ell '}_TL^{r_\ell '}} 
  + \left\| |u|^{2\si_\ell }\nabla u\right\|_{L^{q_\ell '}_TL^{r_\ell '}}\)\\
&\le C_{q_j}\|u_+\|_{H^1} +
C\sum_{\ell=1,2}\|u\|_{L^{\theta_\ell}_TL^{r_\ell}}^{2\si_\ell}\(
\|u\|_{L^{q_\ell}_T L^{r_\ell}} 
  +\|\nabla u\|_{L^{q_\ell}_T L^{r_\ell}} \)\\
&\le C_{q_j}\|u_+\|_{H^1} + C\sum_{\ell=1,2}
\|u\|_{L^{q_\ell}_TL^{r_\ell}}^{2\si_\ell\eta_\ell
  }\|u\|_{L^\infty_TL^{r_\ell}}^{2\si_\ell(1-\eta_\ell) }
  \|u\|_{L^{q_\ell}_T W^{1,r_\ell}}\ , 
\end{align*}
for $\eta_\ell=q_\ell/\theta_\ell\in (0,1]$. Sobolev embedding and the
definition of $X_T$ then imply: 
\begin{align*}
 \left\| \Phi(u)\right\|_{L^{q_j}_T W^{1,2\si_j +2}} \le
 C_{q_j}\|u_+\|_{H^1} +
 C\sum_{\ell=1,2}\left\|u_{\rm
     free}\right\|_{L^{q_\ell}_TL^{r_\ell}}^{2\si_\ell \eta_\ell 
  }\|u\|_{L^\infty_TH^1}^{2\si_\ell(1-\eta_\ell) } \|u\|_{L^{_\ell}_T W^{1,2\si_\ell
  +2}} .
\end{align*}
We have similarly
\begin{align*}
  \left\| \Phi(u)\right\|_{L^\infty_T H^1}&\le  \|u_+\|_{H^1} +
  C\sum_{\ell=1,2}\left\|u_{\rm
      free}\right\|_{L^{q_\ell}_TL^{r_\ell}}^{2\si_\ell\eta_\ell  
  }\|u\|_{L^\infty_TH^1}^{2\si_\ell(1-\eta_\ell) } \|u\|_{L^{q_\ell}_T W^{1,2\si_\ell
  +2}}, \\
\left\| \Phi(u)\right\|_{L^{q_j}_T L^{2\si_j +2}}& \le
  \left\|u_{\rm free}\right\|_{L^{q_j}_TL^{2\si_j+2}} 
+  C\sum_{\ell=1,2}\left\|u_{\rm
      free}\right\|_{L^{q_\ell}_TL^{r_\ell}}^{2\si_\ell\eta_\ell  
  }\|u\|_{L^\infty_TH^1}^{2\si_\ell(1-\eta_\ell) } \|u\|_{L^{q_\ell}_T W^{1,2\si_\ell
  +2}}.
\end{align*}
From Strichartz estimates, $u_{\rm free} \in L^{q_j}(\R;L^{r_j})$, so 
\begin{equation*}
  \left\|u_{\rm free}\right\|_{L^{q_j}_TL^{2\si_j+2}} \to 0\quad
  \text{as }T\to +\infty. 
\end{equation*}
Since $\eta_\ell>0$, we infer that $\Phi$ sends $X_T$ to itself, for
$T$ sufficiently large. 
\smallbreak

We have also, for $u_2,u_1\in X_T$:
\begin{align*}
  \left\| \Phi(u_2)-\Phi(u_1)\right\|_{L^{q_j}_T L^{r_j}}&\lesssim
 \sum_{\ell=1,2} \max_{k=1,2}\| u_k\|_{L^{\theta_\ell}_TL^{r_\ell}}^{2\si_\ell} \left\|
  u_2-u_1\right\|_{L^{q_\ell}_T L^{r_\ell}}\\
&\lesssim \sum_{\ell=1,2}\left\|u_{\rm
    free}\right\|_{L^{q_\ell}_TL^{r_\ell}}^{2\si_\ell\eta_\ell 
  }\|u_+\|_{H^1}^{2\si_\ell(1-\eta_\ell) }\left\|
  u_2-u_1\right\|_{L^{q_\ell}_T L^{r_\ell}}.
\end{align*}
Up to choosing $T$ larger, $\Phi$ is a contraction on $X_T$, equipped
with the $L^{q_1}_T L^{r_1}\cap L^{q_2}_T L^{r_2}$-norm.  Since this
space is complete (see e.g. \cite[Section~4.4]{CazCourant}), 
the proposition follows from the standard fixed point argument.
\end{proof}

Proposition~\ref{prop:scatt} rules out the asymptotic extinction of the solution, since, by invoking the triangle inequality, 
we directly infer
\begin{align*}
\lim_{t\to + \infty} \| u(t)\|_{L^2} \ge & \lim_{t\to + \infty} \big| \| e^{i\frac{t}{2}\Delta}u_+\| _{L^2} - \| u(t)-e^{i\frac{t}{2}\Delta}u_+ \|_{L^2} \big| \\
= &
\lim_{t\to + \infty}  \| e^{i\frac{t}{2}\Delta}u_+\| _{L^2} =   \| u_+\| _{L^2} ,
\end{align*}
due to mass conservation. The solution of the damped equation therefore does not decay to zero, due to the possibility of 
radiation escaping to infinity. The latter is no longer true if we
consider \eqref{eq:nlsp} on a compact manifold, instead of $\R^d$ (see the appendix).

\appendix

\section{Damped NLS on a compact manifold}
\label{sec:manifold}

In the following we shall consider $x\in M$, a $d$-dimensional compact Riemannian manifold without boundary. A particular example is 
$M=\T^d\equiv (\R/2\pi\Z)^d$, the $d$-dimensional torus. Since $M$ is compact by assumption, we do not gain anything from 
the inclusion of a confining potential $V$. We thus set $\omega_j=0$, for all $j=1,\dots, d $ and consider the equation
\begin{equation}
  \label{eq:nlspM}
\left\{
  \begin{aligned}
      i\d_t u + \frac{1}{2}\Delta u & = \l
  |u|^{2\si_1} u -i a |u|^{2\si_2}u,\quad t\ge 0,\ x\in M,\\
\quad u_{\mid t=0}&=u_0\in H^1(M), 
  \end{aligned}
\right.
\end{equation}
where $\Delta$ denotes the Laplace--Beltrami operator on $M$. 

The Hamiltonian analogue of \eqref{eq:nlspM} has been studied in \cite{BGT04}. 
Obviously, we cannot expect the dispersive nature of the free Schr\"odinger group $U(t)$ 
to hold on a compact manifold (a simple counterexample is given by the
eigenfunctions of the  
Laplace--Beltrami operator on $M$). It turns out that this is true even
locally in-time, see Remark~2.6 in \cite{BGT04}. 
The possibility of obtaining Strichartz type estimates on $M$ is therefore severely restricted and 
any proof of a possible bound on $\| U(t) f \|_{L^q(I;L^r(M))}$
requires totally different techniques  
from those needed in the case $M=\R^d$. In view of this, we shall
restrict ourselves to the situation $d\le 3$,  
only, in which case the following local well-posedness result can be
directly deduced from \cite{BGT04} (the case $d=1$ is not treated in
\cite{BGT04}, as it is straightforward using the Sobolev embedding
$H^1(M)\hookrightarrow L^\infty(M)$, that
is, without Strichartz estimates): 

\begin{lemma}[Local well-posedness on compact manifolds]\label{lem:lwpM}
Let $\l,a\in \R$, and assume that it holds
\begin{enumerate}
\item $\si_1, \si_2 > 0$, if $d\le 2$, and
\item $0< \si_1, \si_2\le 1$, if $d=3$. 
\end{enumerate}
Then, for any $u_0 \in H^1(M)$, there exists a time $T>0$ depending
only on $\|u_0\|_{H^1(M)}$  and a unique solution $u\in C([0,T];
H^1(M))$ of \eqref{eq:nlspM}. 
\end{lemma}

\begin{remark}
In the particular case $M=\T^d$ several additional results are available. For example, in $d=1$ local well-posedness in $L^2$ holds 
for nonlinearities which are smaller that quintic \cite{Bou93}. In higher dimensions, though, the situation seems to be more involved
(at least if one seeks for strong solutions in the energy space).
\end{remark}
The proofs given in the case of $\R^d$ readily yield the following result:
\begin{lemma}[A-priori estimates]\label{lem:estcompact}
  Let $d\le 3$, $\l,a\in \R$, $\si_1,\si_2>0$, with $\si_1,\si_2\le 1$
  if $d=3$, and $u_0\in H^1(M)$. On any time intervals $I\ni 0$ such
  that $u\in C(I;H^1(M))$, we have:
  \begin{itemize}
  \item[(1)] Mass dissipation: 
    \begin{equation*}
      \dis  \frac{d}{dt}\| u(t)\|_{L^2}^2 +
 2a\|u(t)\|_{L^{2\si_2+2}}^{2\si_2+2}=0, \quad\forall t\in I.
    \end{equation*}
\item[(2)] Control of the modified energy: for
  $0<\kappa<\frac{a}{\si_2^2+\si_2}$, let 
 \begin{equation*}
E(t)=  \frac{1}{2} \| \nabla u(t)\|_{L^2}^2 +
\frac{\lambda}{\sigma_1+1}  \int_{M}  |u(t,x)|^{2\sigma_1+2} dx +
\kappa \int_{M} |u(t,x)|^{2\sigma_2 +2} \, dx. 
\end{equation*}
If either $\l\ge 0$ or $\l<0$ and $\si_2>\si_1$, there exists a
$C=C(\|u_0\|_{L^2})\ge 0$ such that: 
\begin{equation*}
E(t) \le E(0) +  C(\|u_0\|_{L^2}),\quad \forall t\in I.   
\end{equation*}
 \end{itemize}
\end{lemma}

\begin{proposition}[Decay rate on compact manifolds]\label{prop:extinctM}
  Let $d\le 3$, $\l,a\in \R$, $\si_1,\si_2>0$, with $\si_1,\si_2\le 1$
  if $d=3$, and $u_0\in H^1(M)$. If  either $\l\ge 0$ or $\l<0$ and
  $\si_2>\si_1$, \eqref{eq:nlspM} possesses a unique global solution
  (in the future), $u\in C(\R_+;H^1(M))$. In addition, we
  have
  \begin{equation*}
     \|u(t)\|_{L^2}^2 \le \frac{1}{\(2a t\)^{1/\si_2}|M|}, \quad
     \forall t>0,
  \end{equation*}
and thus $u$ vanishes asymptotically as $t\to +\infty$.  
\end{proposition}
\begin{proof}
  Global existence for positive time stems directly from
  Lemma~\ref{lem:lwpM} and Lemma~\ref{lem:estcompact}. Since $M$ is
  compact,  H\"older inequality yields 
\begin{equation*}
  \|u\|_{L^2(M)}\le |M|^{\si_2/(2\si_2+2)} \|u\|_{L^{2\si_2+2}(M)}.
\end{equation*}
We infer form the equation of mass dissipation,
to get
\begin{equation*}
  \frac{d}{dt}\|u(t)\|_{L^2}^2 + 2a |M|^{\si_2}
  \|u(t)\|_{L^2}^{2\si_2+2}\le 0. 
\end{equation*}
The ODE mechanism sketched in Section~\ref{sec:asymptotic} yields
\begin{equation*}
  \|u(t)\|_{L^2}^2 \le \frac{\|u_0\|_{L^2}^2}{\(1+2a|M|^{\si_2} t \|u_0\|_{L^2}^{2\si_2}\)^{1/\si_2}},
\end{equation*}
hence the result. 
\end{proof}

  The rate established above is the one indicated by the naive ODE argument \eqref{eq:odesol}. Obviously, 
  the decay on $M$ is faster than in the case of partial, or full confinement, but as underlined above,
  none of these rates is claimed to be sharp. 
  
  \begin{remark}  The case
  of compact manifolds corresponds to the one found in numerical simulations (for example, 
  $M=\T^d$ in the case of spectral methods), so the rate of Proposition~\ref{prop:extinctM}
  should be (at least) the one observed numerically. 
  \end{remark}


 \end{document}